\documentclass{birkjour}

\usepackage{amssymb}
\usepackage{bbm}
\usepackage{amsmath}
\usepackage{amsthm}
\usepackage{enumitem}
\usepackage{amsfonts}
\usepackage{hyperref}
\usepackage{comment}
\usepackage{xcolor}
\usepackage{cancel} 

 \newtheorem{thm}{Theorem}
 
 \newtheorem{lem}[thm]{Lemma}
 
 \newtheorem*{thm*}{Theorem}

 \theoremstyle{definition}
 \newtheorem{dfn}[thm]{Definition}
 \theoremstyle{remark}
 \newtheorem{rem}[thm]{Remark}
 
 \newtheorem{qst}{Question}
\newtheorem*{claim}{Claim}

\newtheorem{notation}[thm]{Notation}

\DeclareMathOperator{\CR}{RC}

\newcommand{\NN}{\mathbbm{N}}

\newcommand{\ip}[1]{\left\lfloor #1 \right\rfloor}
\newcommand{\fp}[1]{\left\{#1 \right\}}
\newcommand{\cR}[2][n]{\CR_{#1}(#2)}

\begin{document}

%
%
%
%
%
%
%
%
%

\title[A Note on Iterated Beatty Sequences]
 {A Note on Iterated Beatty Sequences}

\author[M. Khani]{Mohsen Khani}

\address{%
Department of Mathematical Sciences\\ 
Isfahan University of Technology\\ 
Isfahan \\
84156-83111, Iran}

\email{mohsen.khani@iut.ac.ir}

\author[A. N. Valizadeh]{Ali N. Valizadeh}
\address{
	Department of Mathematics and Statistics\\
	University of Isfahan\\
	Isfahan\\
	81746-73441, Iran}
	
\email{a.valizadeh@mcs.ui.ac.ir}

\author[A. Zarei]{Afshin Zarei}
\address{School of Mathematics\\ 
	Institute for Research in Fundamental Sciences 
	(IPM)\\ 
	Tehran\\
	P.O. Box: 19395-5746, Iran.}
\email{a.zarei@ipm.ir}

\thanks{The second author was in part supported by a grant from IPM (No. 1400030021). \\The third author was supported by a grant from IPM.}



%
%
%


\begin{abstract}

 For any irrational number $ \alpha>\frac{3+\sqrt{5}}{2}\approx2.618$ and given a positive $ n\in\NN $, we use elementary number theory to introduce a necessary and sufficient condition for a natural number $ x $ to be in the $n$th iterate of the Beatty sequence of modulus $\alpha$. 
	
\end{abstract}

\maketitle

	\section{Introduction}

	For a given real number $ \alpha $, the Beatty sequence of modulus $ \alpha $ is the range of the function $ f_\alpha:\NN\longrightarrow\NN $ that maps each natural number $ x $ to the integer part of $\alpha x$, denoted by $ \ip{\alpha x} $. We will use $ \fp{\beta} $ to denote the fractional part of a real number $ \beta $, namely $ \fp{\beta}=\beta-\ip{\beta} $. Our main result is the following theorem: 

\begin{thm*}
	Suppose that $ n\geq 1 $, and $ \alpha $ is an irrational number strictly greater than $\frac{3+\sqrt{5}}{2} $. Then, a natural number $ x $ belongs to the 
	range of $ f^n $ if and only if $1-\frac{1}{\alpha}<\fp{\frac{x}{\alpha}}<1$ and for each $i=2, \ldots, n$ we have that
	\begin{align}\label{eq-CRn}
		1-\frac{2}{\alpha}<\fp{\frac{x}{\alpha^{i}}}-\frac{\fp{\frac{x}{\alpha^{i-1}}}}{\alpha}<1-\frac{1}{\alpha}.
	\end{align}
\end{thm*}

	\subsection{A Brief Historical Background}
	Beatty sequences were named after Samuel Beatty's theorem in \cite{Beatty-Solutions}, which, in particular, states that the complement of a Beatty sequence with an irrational modulus $ \alpha $ is again a Beatty sequence whose modulus $ \beta $ is obtained by the equation $ \frac{1}{\alpha}+\frac{1}{\beta}=1 $. 
	
	Beatty sequences have been extensively studied from different perspectives including, but not limited to: Their connection to winning positions of a Wythoff game (see \cite{Connell-WythoffGame}), the impossibility for a Beatty sequence to contain an infinite arithmetic progression (see \cite{Connell-BeattySeqsI} and  \cite{Connell-BeattySeqsII}), disjoint covering systems (see \cite{Uspensky-OnProblemsRaisingFromGame}, \cite{Graham-UspenskyTheorem},  \cite{Fraenkel-TheBracketfunctionAndComplementarySetofIntegers}, \cite{Graham-CoveringPositiveIntegersbyDisjointSets},  \cite{EggletonFraenkelSimpson-BeattyLangfordSequences}, \cite{GinosarYona-AModelForPairsOfBeattySequences}), the way that primes are distributed over the Beatty sequences (see Theorem 9.9 in \cite{Ellison-LesNomresPremiers}, \cite{Harman-PrimesInIntersectionOfBeattySequences}, \cite{Harman-PrimesInBeattySequencesInShortIntervals} and \cite{SteudingTechnau-TheLeastPromeNumber}
), the asymptotic behaviour of the divisor function $ \tau $ on a Beatty sequence (see Theorems I and II in\cite{Abercrombie-BeattySequencesAndMultiplicative}), prime divisors of elements in a Beatty sequence (see \cite{BanksShparlinski-PrimeDivisorsInBeattySequences}), their connection to Sturmian words (see \cite{Stolarsky-BeattySequencesContinuedFractions} and \cite{HieronymiShalitEtAl-SturmianWords}) and finally their connection to semigroup invariants as appeared in \cite{PortaStolarsky-WythoffPairsAsSemigroupInvariants}. An informative yet concise history on Beatty sequences can be found in Section 3.1.2 of \cite{DiaconisZabell-InPraiseOfUspensky}.
	
	%
	Also, from a logical perspective and following Hieronymi's results in \cite{Hieronymi-Interpreting}, \cite{Hieronymi-ExpansionByTwoDiscrete} and \cite{Hieronymi-WhenScalar}, the logical and model theoretic impacts of augmenting the additive group of integers with a Beatty sequence was studied in \cite{KhaniZare-Wythoff}, \cite{KhaniValiZare-ModelCompletenessBeatty} and [27]\footnote{Unpublished reference: [27] M. Khani, A. N. Valizadeh and A. Zarei - “Fibonacci Numbers and Model-Complete Axiomatization of Presburger Arithmetic Expanded with a Beatty Sequence”, preprint (2026).}.

\subsection{Motivation} 

It is folklore that for any irrational $\beta$, a natural number $x$ is in the range of $f_\beta$ if and only if: 
\begin{align}\label{eqn-range-of-f}
	 1-\frac{1}{\beta }<\fp{\frac{x}{\beta }}<1.
\end{align}

In Theorem 1 of \cite{Fraenkel-IteratedBeatty}, it is proved that when $ \alpha $ is taken to be a non-zero algebraic number of degree less than or equal to $ n $, certain linear combinations with integer coefficients containing the iterates $ f^i_\alpha(x) $ (for $ 0\leq i\leq n $) reduce to a linear combination only containing $ f_\alpha(x) $ and $ x $, up to a bounded value of error. 
This theorem implies a range of interesting applications to the study of algebraic numbers (\cite{Fraenkel-IteratedBeatty}).

	%
	
	It is obvious that for two natural numbers $ x $ and $ y $ we have $ f_\alpha(x+y)=f_\alpha(x)+f_\alpha(y)+i $ where $ i $ is equal to $ 0 $ or $ 1 $. Based on this fact, and using a simple induction, one can show that the $ n $-time iteration of $ f_\alpha $, namely the function $ f_\alpha^n $, can be approximated by the function $ f_{\alpha^n} $ up to an error which is uniformly bounded above. That is, there exists a natural number $ K $, depending on $ \alpha $ and $ n $, such that for any $ x\in\NN $ we have:
	\begin{align}\label{eq-fn}
		f_\alpha^n (x)=f_{\alpha^n}(x)-i,
	\end{align}
	where $ i $ is a natural number less than or equal to $ K $. On the other hand, Inequality \eqref{eqn-range-of-f} shows that $x$ belongs to the range of $f_{\alpha^n}$ if and only 
		\begin{align*}
		1-\frac{1}{\alpha^n}<\fp{\frac{x}{\alpha^n}}<1.
	\end{align*}
	Is it possible to use the above inequality along with \eqref{eq-fn} to find a criteria for an $x$ being in the range of $f^n_\alpha$? Our theorem shows that a careful treatment involving all the values $\fp{\frac{x}{\alpha}}, \fp{\frac{x}{\alpha^2}}, \ldots, \fp{\frac{x}{\alpha^n}}$ is necessary to answer this question.

\section{Proof of the Theorem}\label{secProof}
%

\begin{notation}
		For a function $f$, a natural number $x$, and for any $n\geq1$ we write $ x\in\mathcal{R}_{f^n} $ to denote that $ x $ belongs to the range of the function $ f^n $.
\end{notation}

We can use Inequality \eqref{eqn-range-of-f} to easily conclude the following remark.

\begin{rem}\label{rem-range-of-fn} For all irrational numbers $\beta$, all natural number $x$, and all $n\geq1$ we have that:
	\begin{align*}
		x\in\mathcal{R}_{f_{\beta}^{n+1}}&\qquad\text{iff}\qquad x\in\mathcal{R}_{f_\beta} \quad\text{ and }\quad \ip{\frac{x}{\beta }}+1\in\mathcal{R}_{f_\beta^{n}}.
	\end{align*}
	
\end{rem}
From now on, we fix an irrational number $ \alpha>\frac{3+\sqrt{5}}{2} $, and will drop the subscript $ \alpha $ from our notation; so we write $ f(x) $ instead of $ f_\alpha(x) $. Unless stated otherwise, $ x $ and $ n $ will denote natural numbers while $ n $ is assumed to be strictly positive.

\begin{dfn}\label{dfnSeqs}
	Let $v_0$ be the identity function on natural numbers, and 	$\rho_0$ the constant zero function. For each $n\geq 1$, we recursively define the following functions:
	\begin{align*}
		\begin{cases}
			v_n:&\NN\longrightarrow\NN\\
			&x\longmapsto \ip{\frac{v_{n-1}(x)}{\alpha }}+1.
		\end{cases}
		\qquad\text{ and }\qquad
		\begin{cases}
			\rho_{n}:&\NN\longrightarrow(0,\frac{1}{\alpha})\\
			&x\longmapsto \frac{1-\fp{\frac{x}{\alpha ^n}}}{\alpha }.
		\end{cases}
	\end{align*}
\end{dfn}

\vspace*{.3cm}
\begin{notation} \hfill
	\begin{itemize}
		\item[(1)] The following inequality is denoted by $ \cR{x} $\footnote{ RC stands for the ``Range Condition''.}.
		\[ 0<\rho_{n}(x)-\frac{\rho_{n-1}(x)}{\alpha }<\frac{1}{\alpha ^{2}}. \]
		
		\item[(2)] We use $ \cR[\leq n]{x} $ to denote the conjunction of the conditions $ \cR[i]{x} $ for all $ 1\leq i\leq n $.

	\end{itemize}
	
\end{notation}

By making use of the notation above, our main theorem claims that:
\[ x\in\mathcal{R}_{f^n} \quad\text{iff}\quad \cR[\leq n]{x}.\]

\vspace*{.5cm}
\begin{lem}\label{lmaVn}
	For any $ n $ and $ x $ we have 
	\begin{align*}
		\frac{v_n(x)}{\alpha}=\frac{x}{\alpha^{n+1}}+\sum_{i=1}^{n}\frac{1-\fp{\frac{v_{n-i}(x)}{\alpha}}}{\alpha^{i}}.
	\end{align*}
\end{lem}
\begin{proof}
	By a simple induction on $n$.
\end{proof}

\begin{rem}
	Note, in particular, that if $ x $ belongs to the range of $ f^{n} $, then $f^{-n}(x)=v_n(x)$. 
	Moreover, using the definitions, we have that: 
	\[ v_1(x)=\ip{\frac{x}{\alpha}}-\alpha\rho_1(x). \]
	Also, the equation appearing in Lemma \ref{lmaVn} can be recast as:
	\[ v_n(x)=\frac{x}{\alpha^{n}}+\sum_{i=1}^{n}\frac{\rho_1(v_{n-i}(x))}{\alpha^{i-2}}. \]
	Intuitively speaking, the values of $ \rho_{n}(x) $ are sort of a remainder which, combined with different powers of $ \alpha $, captures the difference between the values of $ f^{-n}(x)$ and $ \ip{\frac{x}{\alpha^n}} $.  Whenever $ x $ belongs to the range of $ f_{\alpha^n} $, the latter value would be equal to $ f^{-1}_{\alpha^n}(x)-1 $. This can be seen using Inequality \eqref{eqn-range-of-f} applied for $\beta=\alpha^n$. 
\end{rem}

%
%
%
%

\begin{lem}[Technical Lemma]\label{lmaTech}\hfill
	
	Suppose that $ n\geq 2 $ and the condition $ \cR[<n]{x} $ holds. Then, either of the conditions (i) or (ii) below implies that:
	\begin{align}\label{eq-lmaTech}
		\rho_{n-1}(v_1(x))-\frac{\rho_{n-2}(v_1(x))}{\alpha }=\rho_{n}(x)-\frac{\rho_{n-1}(x)}{\alpha }.
	\end{align}
	
	\begin{itemize}
		\item[(i)] $ \cR[n]{x} $.
		\item[(ii)] $ x\in\mathcal{R}_{f^{n}} $.
	\end{itemize}
\end{lem}

\begin{proof}
	The following equality is evident using the definitions of $ v_1(x)  $ and $ \rho_1(x) $:
	\begin{align}\label{eqGeneral}
		\frac{v_1(x)}{\alpha ^{n-1}}=\frac{x}{\alpha ^{n}}+\frac{\rho_1(x)}{\alpha ^{n-2}}.
	\end{align}
	
	\begin{claim}
	Either of the conditions (i) or (ii) implies that:
	\begin{align}\label{eqCalim}
		\fp{\frac{v_1(x)}{\alpha ^{n-1}}}=\fp{\frac{x}{\alpha ^{n}}}+\frac{\rho_1(x)}{\alpha ^{n-2}}. 
	\end{align}
\end{claim}

\textit{Proof of the Claim.} Note that $ \alpha>1 $, and the values of $ \fp{\frac{x}{\alpha ^{n}}} $ and $ \frac{\rho_1(x)}{\alpha ^{n-2}} $ belong to the intervals $ (0,1) $ and $ (0,\frac{1}{\alpha^{n-1}}) $ respectively. Therefore, if the claim fails to hold, by Equation (\ref{eqGeneral}), there only remains one possibility for these values:
\begin{align}\label{eqCalimFails}
	\fp{\frac{v_1(x)}{\alpha ^{n-1}}}=\fp{\frac{x}{\alpha ^{n}}}+\frac{\rho_1(x)}{\alpha ^{n-2}}-1.
\end{align}
At the same time, the condition $ \cR[1]{x} $ implies that $ \rho_1(x)\in(0,\frac{1}{\alpha ^{2}}) $, which leads to having: 
\begin{align}\label{eqRho1}
	0<\frac{\rho_1(x)}{\alpha ^{n-2}}<\frac{1}{\alpha^{n}}.
\end{align}

Now, we proceed by induction on $ n $ to show that (\ref{eqCalimFails}) is contradictory for each $ n $.

For $ n=2 $, assume (i), that is $ \cR[2]{x} $, then, using (\ref{eqCalimFails}) we immediately have that $ \fp{\frac{v_1(x)}{\alpha ^{n-1}}}< 0 $, which is impossible for the fractional part of a real number. 

If we assume (ii), by Remark \ref{rem-range-of-fn}, we have that: 
\[ 1-\frac{1}{\alpha}<\fp{\frac{v_1(x)}{\alpha}}<1. \]
This, together with (\ref{eqCalimFails}) and (\ref{eqRho1}), shows that $\fp{\frac{x}{\alpha^2}}$ is strictly greater than $ 2-(\frac{1}{\alpha}+\frac{1}{\alpha^2}) $. But, the latter value is strictly greater than 1, as $ \alpha>\frac{3+\sqrt{5}}{2} $ implies that $\alpha$ is strictly greater than the roots of the polynomial $\alpha^2-\alpha-1$. That is, $ \fp{\frac{x}{\alpha^2}}>2-(\frac{1}{\alpha}+\frac{1}{\alpha^2})>1 $, which is impossible for the fractional part of a real number.

Now, suppose that $ n>2 $ and assume (i). By the induction hypothesis, (\ref{eqCalim}) holds for $ n-1 $. This equation can be rewritten as: 
\[ \frac{v_1(x)}{\alpha ^{n-2}}- \ip{\frac{v_1(x)}{\alpha ^{n-2}}}=\fp{\frac{x}{\alpha ^{n-1}}}+\frac{\rho_1(x)}{\alpha ^{n-3}}. \]
For better readability, let $ \gamma $ denote $ \frac{v_1(x)}{\alpha ^{n-2}} $, and observe that the equality above, when divided by $ \alpha $, is equivalent to: 
\[ \frac{\gamma}{\alpha}-\frac{\ip{\gamma}+1-1}{\alpha}=\frac{\fp{\frac{x}{\alpha ^{n-1}}}}{\alpha}+\frac{\rho_1(x)}{\alpha ^{n-2}}. \]
By replacing $ \frac{\gamma}{\alpha} $ with $ \ip{\frac{\gamma}{\alpha}} + \fp{\frac{\gamma}{\alpha}} $, and $ \frac{\ip{\gamma}+1}{\alpha} $  with $ \ip{\frac{\ip{\gamma}+1}{\alpha}} + \fp{\frac{\ip{\gamma}+1}{\alpha}} $, we obtain
\[ \ip{\frac{\gamma}{\alpha}}-\ip{\frac{\ip{\gamma}+1}{\alpha}}+\fp{\frac{\gamma}{\alpha}}+\frac{1}{\alpha}-\frac{\fp{\frac{x}{\alpha ^{n-1}}}}{\alpha}-\frac{\rho_1(x)}{\alpha ^{n-2}}=  \fp{\frac{\ip{\gamma}+1}{\alpha}}. \]

It is easy to see that $ \ip{\frac{\gamma}{\alpha}}-\ip{\frac{\ip{\gamma}+1}{\alpha}} $	is an integer less that or equal to zero; let $ \ell $ denote this integer. Also, let $ \varepsilon $ denote the real number $ \fp{\frac{\ip{\gamma}+1}{\alpha}} $ which by definition belongs to the interval $ (0,1) $.

We can use (\ref{eqGeneral}) and the fact that $ \gamma=\frac{v_1(x)}{\alpha ^{n-2}} $ to replace $ \fp{\frac{\gamma}{\alpha}} $ with its original value. Then, by the definition of $ \rho_{n-1}(x) $ and by doing some elementary calculations we will have that:
\begin{align}\label{eqEllEpsilon}
	\ell-\varepsilon = 1-\fp{\frac{x}{\alpha^{n}}}-\rho_{n-1}(x).
\end{align}

The condition $ \cR{x} $ is equivalent to: 
\[ 0<1-\fp{\frac{x}{\alpha^{n}}}-\rho_{n-1}(x)<\frac{1}{\alpha}, \]
which implies that:
\[ \ell-\frac{1}{\alpha}<\varepsilon<\ell\leq 0. \]
However, this cannot happen for $ \varepsilon $ as the fractional part of a real number, which must be positive. 

Now, assume (ii), and note that 
by Lemma \ref{lmaVn} we have that: 
\begin{align}\label{eqSum}
	\frac{x}{\alpha^{n}} = -\frac{v_{n-1}(x)}{\alpha}+\sum_{i=1}^{n-1}\frac{1-\fp{\frac{v_{n-1-i}(x)}{\alpha}}}{\alpha^{i}}.	
\end{align}
Let $ \delta $ denote the summation appeared above. Since $ x $ belongs to the range of $ f^n $, by Remark \ref{rem-range-of-fn}, for each $ 1\leq i\leq n-1 $ we have that 
\[ 1-\frac{1}{\alpha}<\fp{\frac{v_{i}(x)}{\alpha}}<1. \]
This implies that: 
\[ 0<\delta<\sum_{i=2}^{n}\frac{1}{\alpha^i}. \]
In particular, we have that: 
\begin{align}\label{eqDeltaVn-1}
	-1<\delta-\fp{\frac{v_{n-1}(x)}{\alpha}}<-1+\sum_{i=1}^{n}\frac{1}{\alpha^i}.
\end{align}
Now, note that the summation $ \sum_{i=1}^{n}\frac{1}{\alpha^i} $ is strictly less than the sum of the infinite series $ \sum_{i=1}^{\infty}\frac{1}{\alpha^i} $ which exists and is equal to $ \frac{1}{\alpha-1} $. Since $ \alpha>2 $, the value of $ \sum_{i=1}^{n}\frac{1}{\alpha^i}-1 $ is strictly less than $ 0 $. Therefore, the fractional part of the number $ \delta-\frac{v_{n-1}(x)}{\alpha} $ is equal to:
\[ \delta-\fp{\frac{v_{n-1}(x)}{\alpha}}+1. \]
Hence, by (\ref{eqSum}), we have that: 
\[ \fp{\frac{x}{\alpha^{n}}}=\delta-\fp{\frac{v_{n-1}(x)}{\alpha}}+1, \]
which, using (\ref{eqCalimFails}), implies that: 
\[ \fp{\frac{v_1(x)}{\alpha ^{n-1}}}=\delta-\fp{\frac{v_{n-1}(x)}{\alpha}}+\frac{\rho_1(x)}{\alpha ^{n-2}}. \]
Using the Inequalities (\ref{eqRho1}) and (\ref{eqDeltaVn-1}) we have that: 
\[ \fp{\frac{v_1(x)}{\alpha ^{n-1}}}<-1+\sum_{i=1}^{n}\frac{1}{\alpha^i} + \frac{1}{\alpha^n}. \]
However, the right-hand side of the latter inequality is strictly less than 
\[ -1+\sum_{i=1}^{\infty}\frac{1}{\alpha^i} + \frac{1}{\alpha^3},  \] 
or equivalently, it is less than  
$ -1+\frac{1}{\alpha-1}+\frac{1}{\alpha^3} $. The latter value is strictly negative, as the roots of the polynomial $\alpha^4-2\alpha^3-\alpha +1$ are strictly less than $\frac{3+\sqrt{5}}{2}$. Hence, the fractional part of $ \frac{v_1(x)}{\alpha ^{n-1}} $ would be negative, which is impossible.

\hfill \scalebox{.8}{$ \blacksquare_{Claim} $}

The claim above directly proves the case of $ n=2 $. For $ n\geq 3 $, by definitions, we have that:
\[ \rho_{n-1}(v_1(x))-\frac{\rho_{n-2}(v_1(x))}{\alpha }=\frac{1}{\alpha }-\frac{1}{\alpha ^{2}}-\frac{\fp{\frac{v_1(x)}{\alpha ^{n}}}}{\alpha }+\frac{\fp{\frac{v_1(x)}{\alpha ^{n-1}}}}{\alpha ^{2}}. \]

Now, if either of the conditions (i) or (ii) holds for $ n $, the same condition will automatically hold for $ n-1 $. Hence, we can apply the claim above  for $ n $ and $ n-1 $ simultaneously to recalculate  the right-hand side of the equation above, to obtain 	
\[ \frac{1}{\alpha }-\frac{1}{\alpha ^{2}}-\frac{\fp{\frac{x}{\alpha ^{n}}}}{\alpha }+\frac{\fp{\frac{x}{\alpha ^{n-1}}}}{\alpha ^{2}}, \]
which would be the same as $ \rho_n(x)-\frac{\rho_{n-1}(x)}{\alpha }. $
\end{proof}


Note that Equation \eqref{eq-lmaTech} in the lemme above enables us to omit $ v_1(x) $ from the left of the equation just by increasing the indices of the functions $ \rho $ on the right-hand side of the equation.
\subsection{Proof of the Theorem}
	We proceed by induction on $ n. $ The case of $ n=1 $ directly follows from Inequality \eqref{eqn-range-of-f}. For $ n\geq 2 $, first suppose that $ x $ is in the range of $ f^{n} $. By induction hypothesis, we know that the condition $ \cR[<n]{x} $ holds. Hence, it suffices to prove $ \cR{x} $. By Remark \ref{rem-range-of-fn}, we know that $ v_1(x) $ belongs to the range of $ f^{n-1} $. By applying the induction hypothesis for $ v_1(x) $, it is implied that the condition $ \cR[n-1]{v_1(x)} $ holds for $v_1(x)$. Namely, we have the following inequality:
	\[ 0<\rho_{n-1}(v_1(x))-\frac{\rho_{n-2}(v_1(x))}{\alpha }<\frac{1}{\alpha ^{2}}. \] 
	On the other hand, by Lemma \ref{lmaTech}, we have that 
	\[ \rho_n(x)-\frac{\rho_{n-1}(x)}{\alpha } =\rho_{n-1}(v_1(x))-\frac{\rho_{n-2}(v_1(x))}{\alpha }, \]
	which, using the inequality above, proves $ \cR[n]{x}. $
	
	Now, suppose that $ \cR[\leq n]{x} $ holds. Based on Remark \ref{rem-range-of-fn}, we only need to show that $ x $ and $ v_1(x) $ belong respectively to the range of $ f $ and $ f^{n-1} $. By the induction hypothesis, $ x $ belongs to the range of $ f^{n-1} $, which, by Remark \ref{rem-range-of-fn}, implies that $ v_1(x) $ is in the range of $ f^{n-2} $. By applying the induction hypothesis for $ v_1(x) $, we conclude that the condition $ \cR[\leq {n-2}]{v_1(x)} $ holds. Hence, we only need to prove $ \cR[n-1]{v_1(x)} $. But, by $ \cR{x} $, we have that: 
	\[ 0<\rho_{n}(x)-\frac{\rho_{n-1}(x)}{\alpha }<\frac{1}{\alpha ^{2}}, \] 
	which, using Lemma \ref{lmaTech}, implies the desired condition $ \cR[n-1]{v_1(x)} $. \hfill\qed

\begin{rem}
	
	In view of the results appeared in \cite{KhaniZare-Wythoff}, \cite{KhaniValiZare-ModelCompletenessBeatty}, and [27] if $ \alpha $ happens to be an algebraic number of degree $ m $, then the conditions \eqref{eq-CRn} in the Theorem reduce to inequalities which only involve the powers of $ \alpha $ strictly less than $ m $. This observation goes parallel with the fact that in the logical and model theoretic setting that studies the additive group of integers while allowing the function $ f_\alpha $ to be in the language, all first-order terms reduce to linear combinations of iterates of $ f_\alpha $ up to the power of $m-1$: $f_\alpha,f^2_\alpha, \ldots, f^{m-1}_\alpha $. This fact has important consequences from a model-theoretic perspective since it provides a possibility for having quantifier-elimination in relatively a small language. One such impact can be seen in \cite{KhaniZare-Wythoff} and [27] where  $\alpha$ is assumed to be the golden ratio as a typical quadratic number. This is basically in contrast to the case of a transcendental number $ \alpha $ where coping with all the powers of $ \alpha $ seems unavoidable in treating the iterates of $ f_\alpha $ as well as the complexity of formulas (see \cite{KhaniValiZare-ModelCompletenessBeatty}).
\end{rem}
We close by asking two questions:
\begin{qst}
	Our assumption on $\alpha$ being strictly greater than $\frac{3+\sqrt{5}}{2}\approx2.618$ implies that it satisfies the algebraic inequalities below:
	\[ \alpha^{2}-3\alpha+1>0 \quad\text{ and }\quad \alpha^{4}-2\alpha^{3}-\alpha+1>0. \]
	In fact, the greatest root of the polynomials above is approximately $ 2.618 $ and $ 2.118 $, respectively. As we observed, these inequalities are crucial to prove Lemma \ref{lmaTech}. A natural question is whether our theorem holds also for irrationals belonging to the following set:
	\[ (1, 2.618\hspace*{-3pt}\sim)=(1, 2.118\hspace*{-3pt}\sim)\cup(2.118\hspace*{-3pt}\sim, 2.618\hspace*{-3pt}\sim). \]
	
In this regard, recall that the function $ f_\alpha $ is surjective whenever $ \alpha $ is less than or equal to one.
	
\end{qst}


\begin{qst}
	Is it possible to generalize the result of this paper to iteration of Beatty sequences with a distinct set of moduli? In other words, are there similar conditions for iterations of the form $ (f_\alpha\circ f_\beta)^n $ or $ (f_\beta\circ f_\alpha)^n $ where $ \alpha $ and $ \beta $ are taken to be distinct irrational numbers? It seems plausible that possible algebraic dependencies that can occur between $\alpha$ and $\beta$ will play a role in finding such conditions.
\end{qst}



\providecommand{\bysame}{\leavevmode ---\ }
\providecommand{\og}{``}
\providecommand{\fg}{''}
\providecommand{\smfandname}{and}
\providecommand{\smfedsname}{\'eds.}
\providecommand{\smfedname}{\'ed.}
\providecommand{\smfmastersthesisname}{M\'emoire}
\providecommand{\smfphdthesisname}{Th\`ese}

\end{document}